\documentclass[reqno]{amsart}
\usepackage{mathrsfs}
\usepackage{amsmath}
\usepackage{amsthm}
\usepackage{amsfonts}
\usepackage{amssymb}
\usepackage{url}
\usepackage{enumerate}
\usepackage[pdftex,bookmarks=true]{hyperref}
  \usepackage[usenames, dvipsnames]{color}
\usepackage{verbatim}

\usepackage{pdfsync}

\newcommand\R{{\mathbf{R}}}

\renewcommand\P{{\mathbf{P}}}

\newcommand\dist{{\operatorname{dist}}}
\newcommand\Z{{\mathbf{Z}}}

\newcommand\I{{\mathbf{I}}}

\newcommand\Bv{{\mathbf v}}
\newcommand\Bw{{\mathbf w}}
\newcommand\Bx{{\mathbf x}}
\newcommand\By{{\mathbf y}}

\newcommand\BP{{\mathbf P}}

\newcommand\BR{{\mathbf R}}
\newcommand\BS{{\mathbf S}}

\newcommand\eps{\varepsilon}

\usepackage{fullpage}

\theoremstyle{plain}
  \newtheorem{theorem}[subsection]{Theorem}

    \newtheorem{proposition}[subsection]{Proposition}
  
  \newtheorem{lemma}[subsection]{Lemma}
  \newtheorem{corollary}[subsection]{Corollary}

\theoremstyle{definition}
  \newtheorem{definition}[subsection]{Definition}

\def \EE {\mathcal{E}}

\begin{document}

\title{Eigenvalue Gaps of Random Perturbations of Large Matrices}

\author{Kyle Luh}
\address{Department of Mathematics\\ University of Colorado Boulder}
\email{kyle.luh@colorado.edu}
\thanks{K. Luh was supported in part by the Ralph E. Powe Junior Faculty Enhancement Award.}

\author{Ryan Vogel}
\address{Department of Mathematics\\ University of Colorado Boulder}
\email{ryan.vogel@colorado.edu}
\thanks{R. Vogel was supported in part by the CU Boulder REU program.}

\author{Alan Yu}
\address{Department of Mathematics\\ University of Colorado Boulder}
\email{alan.yu@colorado.edu}
\thanks{A. Yu was supported in part by the CU Boulder REU program.}




\begin{abstract} 
The current work applies some recent combinatorial tools due to Jain to control the eigenvalue gaps of a matrix $M_n = M + N_n$ where $M$ is deterministic, symmetric with large operator norm and $N_n$ is a random symmetric matrix with subgaussian entries.  One consequence of our tail bounds is that $M_n$ has simple spectrum with probability at least $1 - \exp(-n^{2/15})$ which improves on a result of Nguyen, Tao and Vu in terms of both the probability and the size of the matrix $M$.  
 \end{abstract}

\maketitle

\section{Introduction }
The gaps between eigenvalues of a random matrix have occupied probablists since the introduction of random matrix theory \cite{wigner1951statistical, wishart1928generalised}. These gap statistics have been studied for a large set of models and with a wide variety of techniques (see \cite{forrester2000wigner, erdos2010wegner, tao2011universality, erdos2011universality,arous2013extreme, tao2013singlegap} and the references therein). An interesting question, originally posed by the computer science community in relation to the notorious graph isomorphism problem, was to obtain estimates on the smallest gap size, $\delta_{min}$, of a random matrix \cite{babai1982isomorphism}.  In fact, at the time, it was not known whether the gap was non-zero with high probability.  This fact was first established in \cite{tao2017simple}.  This result followed shortly by a more quantitative estimate of the gap size in \cite{nguyen2017tailbounds}.

In the present work, we examine $n \times n$ random matrices of the form $M_n = M + N_n$ where $M$ is a deterministic, symmetric matrix and $N_n$ is a symmetric random matrix with centered entries.  In  \cite{nguyen2017tailbounds}, the authors obtained some quantitative estimates on the gap sizes of $M_n$ when $M$ has operator norm that is at most polynomial in $n$, the size of the matrix.  

\begin{theorem}\cite[Theorem 2.6]{nguyen2017tailbounds}
	Let $N_n$ be populated with centered, subgaussian random variables of variance 1.  
	If $\|M\| \leq n^{c}$ for some constant $c > 0$, then for any fixed $C > 0$, there exists a $C'$ such that
	\[
	\P(\delta_{min} \leq n^{-C'}) \leq n^{-C}.
	\]
\end{theorem}
A simple consequence of this theorem is that $N_n$ has simple spectrum with probability at most $n^{-C}$ for any $C > 0$.  In the present work, we build on some recent combinatorial techniques in random matrix theory \cite{ferber2021counting, ferber2021resilience, ferber2019symmetric,  jain2021approximate, jain2021quantitative, jain2021strong,  luh2021finitefields}.  Using this method, we provide the first tail bounds for gap sizes of $M_n$ when $M$ can have operator norm exponential in the size of the matrix.  An immediate consequence of our result is an improved estimate on the probability of having simple spectrum that is exponentially small in $n$, rather than polynomial.  The rest of the article is devoted to the proof of our main theorem.

\begin{definition}
	A random variable $X$ is \emph{subgaussian} if there exists a constant $K > 0$ so that for all $t > 0$, 
	\[
	\P(|X| \geq t) \leq 2 \exp\left(-\frac{t^2}{K^2} \right).
	\]
\end{definition}

\begin{theorem} \label{thm:main}
Let $M_n = M + N_n$ where $\|M\| \leq \exp(n^{1/16})$.  We define $\lambda_i := \lambda_{i}(M_{n})$ where
\[
\lambda_1 \leq \lambda_2 \leq \cdots \leq \lambda_n.
\] We let $\delta_i = \lambda_{i+1} - \lambda_i$ and $\delta_{min} = \min_i \delta_i$.  
For $\alpha \geq  \exp(-n^{2/15})$ and
 $\nu =  (C_{\ref{thm:main}}\|M\| \alpha^{-1} n^{7/6})^{-4 \frac{\log(\alpha^{-1})}{\log n}}$,
\[
\P(\delta_{min} \leq \nu) \leq \alpha.
\]
\end{theorem}  

Specializing the above theorem to a gap size of zero yields a bound on the probability of having simple spectrum.
\begin{corollary}
	Under the hypotheses of Theorem \ref{thm:main}, the probability that $M_n$ has simple spectrum is greater than $1 - \exp(-n^{2/15})$.  
\end{corollary}

\section{Proof Strategy} \label{sec:strategy}
We begin with the strategy first proposed in \cite{tao2017simple,nguyen2017tailbounds}.  We decompose $\mathbf{M}_n$ as 
\begin{equation}
	M_n = \left( \begin{array}{cc} M_{n-1} & \mathbf{x} \\
		\mathbf{x}^T & m_{nn} \end{array} \right)
\end{equation} 
where $M_{n-1}$ is the $n-1 \times n-1$ matrix in the upper left, $\mathbf{x}$ is an $n-1 \times 1$ vector and $m_{nn}$ is the remaining random variable in the lower right.  

By definition, for the $i$-th eigenvalue $\lambda_i(M_n)$ with unit eigenvector $\mathbf{v}$,
\[
\left( \begin{array}{cc} M_{n-1} & \mathbf{x} \\
	\mathbf{x}^T & m_{nn} \end{array} \right)  \left(\begin{array}{c} \mathbf{v}' \\ v_n \end{array} \right) = \lambda_i(M_n) \left(\begin{array}{c} 	\mathbf{v}' \\ v_n \end{array} \right) \text{ where we have written } \mathbf{v} = \left(\begin{array}{c}	\mathbf{v}' \\ v_n\end{array} \right).
\]
If we examine the top $n-1$ coordinates, we find that
\[
(M_{n-1} - \lambda_i(M_n)) \mathbf{v}' + v_n \mathbf{x} = 0.
\]
Now consider taking the innerproduct of the above equation with the eigenvector $\mathbf{w}$ corresponding to $\lambda_i(M_{n-1})$, the $i$-th eigenvalue of $M_{n-1}$. We obtain 
\begin{equation} \label{eq:gapsimplyinnerproduct}
|v_ n \mathbf{w}^T \mathbf{x}| = |(\lambda_i(M_{n-1}) - \lambda_i(M_n)| |\mathbf{w}^T \mathbf{v}'| \leq |(\lambda_i(M_{n-1}) - \lambda_i(M_n)|.
\end{equation}
By Cauchy's interlacing law, for any $\eta > 0$, the event $\lambda_{i+1}(M_n) - \lambda_i(M_n) \leq \eta$ implies the event $|(\lambda_i(M_{n-1}) - \lambda_i(M_n)| \leq \eta$.  This, in turn, by \eqref{eq:gapsimplyinnerproduct}, implies the event $|v_n \mathbf{w}^T \mathbf{x}| \leq \eta$.  We cannot guarantee that $v_n$ is large or even non-zero, but since $\mathbf{v}$ is a unit vector, there must be some coordinate that is of size greater than $n^{-1/2}$.  Taking a union bound over all $[n]$, we can assume that $|v_n| \geq n^{-1/2}$.  Thus, we need to bound the probability that $|\mathbf{w}^T \mathbf{x}| \leq \eta n^{1/2}$.  

Crucially, $\mathbf{w}$ and $\mathbf{x}$ are independent.  Therefore, the problem reduces to an issue of anti-concentration.  If we condition on $\mathbf{w}$, this falls under the domain of Littlewood-Offord theory \cite{littlewood1943roots}.  In a long series of works, it has been established that the anti-concentration of $\mathbf{w}^T \mathbf{x}$ is determined by the arithmetic structure of $\mathbf{w}$.  The bulk of the argument now is to show that eigenvectors of random matrices are disordered.  In random matrix theory, this is usually established by Inverse Littlewood-Offord theory \cite{tao2009inverselittlewood} or covering arguments via the Least Common Denominator \cite{rudelson2008invertibility}. Both of these breakthrough ideas have some drawbacks. The arguments that invoke Inverse Littlewood-Offord theory are only capable of providing inverse polynomial bounds on the probability of various events.  The covering arguments, on the other hand, are sensitive to the operator norm of the random matrix and tend to break down when the norm is too large. Although the Inverse Littlewood-Offord theory is more robust to the size of the operator norm, even this method can only handle norms polynomial in the size of the matrix.  A recent combinatorial innovation \cite{ferber2021counting} provides a method of bypassing the inverse theorems and extracting super-polynomial probability bounds.  This method has been applied successfully to many discrete random matrix questions \cite{ferber2019symmetric, luh2021finitefields, ferber2021resilience, jain2021approximate, jain2021strong}.  In \cite{jain2021quantitative}, Jain introduced a method of lattice approximations to extend least singular value bounds to matrices of superpolynomial norm perturbed by i.i.d. random matrices.  In our argumement, we adapt Jain's lattice approximation to the symmetric random matrix.  The key idea is to approximate subsets of a vector so that when we restrict our attention to those columns of the random matrix, there are many completely independent rows.  This is an insight that dates back to \cite{vershynin2014symmetric}.

\section{Auxiliary Results}
In this section, we consolidate some basic results that will be useful in the proof of Theorem \ref{thm:main}.
\subsection{Concentration of Random Variables}
We begin with a more precise definition of subgaussian.  
\begin{definition}
	A random variable is $K$-\emph{subgaussian} if for all $t > 0$, 
	\[
	\P(|X| \geq t) \leq 2 \exp\left(-\frac{t^2}{K^2} \right).
	\]
\end{definition}
The next lemma indicates that the linear combination of subgaussian random variables behaves well.  
\begin{lemma} \cite[Theorem 2.6.3]{vershynin2018highdimensional}
	Let $X_1, \dots, X_n$ be independent, mean-zero, $K$-subgaussian random variables and let $\mathbf{a} \in \mathbb{R}^n$.  Then, for all $t \geq 0$, 
	\[
	\P\left( \left| \sum_{i=1}^n a_i X_i\right| \geq t \right) \leq 2 \exp\left(- \frac{c t^2}{K^2 \|a\|_2^2} \right)
	\]
	for some absolute constant $c > 0$. 
\end{lemma}
The following result is a well-known bound on the operator norm of a random matrix.  
\begin{lemma}\cite[Theorem 4.4.5]{vershynin2018highdimensional} \label{lem:opnorm}
	Let $N_n$ be a symmetric $n \times n$ matrix whose entries on and above the diagonal are independent, mean-zero, $K$-subgaussian random variables.  Then, 
	\[
	\P(\|N_n\| \geq C_{\ref{lem:opnorm}} \sqrt{n}) \leq 2 \exp(-n)
	\]
	where $C_{\ref{lem:opnorm}}$ depends only on $K$.  
\end{lemma}

\subsection{Anti-concentration of Random Variables}
To capture the notion of anti-concentration, we introduce the following definition.
\begin{definition}
	For a vector $\mathbf{x} \in \R^n$ and random variable $\xi$, we define the \emph{L\'evy concentration function} of $\mathbf{x}$ to be 
	\[
	\rho_{\xi}(\mathbf{x}, \eps) = \sup_{\mathbf{y} \in \R} \P(|\langle \mathbf{x}, \boldsymbol{\xi} \rangle - \mathbf{y}| \leq \eps)
	\]
	where $\boldsymbol{\xi} \in \R^n$ is a vector of independent copies of $\xi$.  
\end{definition}

	As we will only be dealing with the random variables from Theorem \ref{thm:main}, we will often omit the dependence on $\xi$ from the notation from now on so that $\rho(\Bx) := \rho_{\xi}(\Bx)$.

\begin{lemma}\cite[Corollary 6.3]{tao2010smoothanalysis} \label{lem:constant}
	Let $\boldsymbol{\xi} = (\xi_1, \dots, \xi_n)$ be a random vector with independent, identical copies of a random variable, $\xi$, which has variance 1.  Then there exist constant $c \in (0,1)$ such that
	\[
	\sup_{\mathbf{v} \in \mathcal{S}^{n-1}} \rho_{\xi}( \mathbf{v}, c_{\ref{lem:constant}}) \leq 1 - c_{\ref{lem:constant}}.
	\]
\end{lemma}

The next lemma is from \cite{jain2021quantitative} and establishes that the L\'evy concentration of an inner product of a random vector with a deterministic vector is stable under perturbations of the deterministic vector.  
\begin{lemma} \cite[Proposition 3.4]{jain2021quantitative}\label{lem:stability}
	Let $\mathbf{y}, \mathbf{z} \in \R^n$ and $r, s \geq 0$.  Then, for $\boldsymbol{\xi} = (\xi_1, \dots, \xi_n)$, a random vector with independent, mean-zero, $K$-subgaussian random variables, 
	\[
	\rho( \mathbf{y}, r+s) \geq \rho( \mathbf{z} , r) - e \cdot \exp\left( - \frac{c_{\ref{lem:stability}} s^2}{\|\mathbf{y} - \mathbf{z}\|_2^2} \right)
	\]
\end{lemma}

The next lemma indicates the relationship between the anti-concentration of a vector and a subvector.
\begin{lemma} \cite[Lemma 2.1]{rudelson2008invertibility} \label{lem:restriction}
	Let $\Bv \in \R^n$ and $\Bv' \in \R^m$ where $m \leq n$ and $\Bv'$ is a subvector of $\Bv$, meaning its entries are a subset of those of $\Bv$.
	Then, for $r > 0$, 
	\[
	\rho(\Bv, r) \geq \rho(\Bv', r).
	\]  
\end{lemma}

The last proposition of this section guarantees that a vector that does not have a scaling near an integer lattice point must have small L\'evy concentration function.
\begin{proposition} \cite[Proposition 2.8]{jain2021quantitative}\label{prop:distance}
	Let $\xi_1, \dots, \xi_n$ be independent copies of a $K$-subgaussian random variable $\xi$ and let $\mathbf{y} \in \R^n \setminus \{\mathbf{0}\}$.  Suppose that there exists $\alpha(n) \in (0,1)$ and $\beta(n) \in (1, \infty)$ such that
	\[
	\dist(\gamma \mathbf{y}, \mathbb{Z}^n) \geq \tau 
	\]
	for all $\gamma \in \R$ such that $|\gamma| \in [\alpha(n), \beta(n)]$.  Then, for any $r \geq 0$, 
	\[
	\rho_{\xi}(\By, r) \leq C_{\ref{prop:distance}} \exp(\pi r^2) \Big( \exp(-c_{\ref{prop:distance}} \beta(n)^2) + \exp(-c_{\ref{prop:distance}} \tau^2) + \alpha(n) \Big)
	\] 
	where $C_{\ref{prop:distance}}, c_{\ref*{prop:distance}}$ depend only on $K$.  
\end{proposition}

\subsection{Counting Lemma}
Our last auxiliary result is the key combinatorial tool that allows us to improve on the Inverse Littlewood-Offord theorems that are used in \cite{nguyen2017tailbounds}.
\begin{theorem} \cite[Theorem 2.8]{jain2021quantitative} \label{thm:counting}
	For $\rho \in (0,1)$, $\xi$ a random variable with variance 1, we define
	\[
	\mathbf{S}_{\rho} := \left \{ \mathbf{a} \in \mathbb{Z}^n: \rho_{\xi}(\mathbf{a}, 1) \geq \rho\right\}.
	\]
	There exists a constant $C_{\ref{thm:counting}}$ such that for $l, m \in \mathbb{N}$ with $1000 K \leq l \leq \sqrt{m} \leq m \leq n/\log n$, the following holds.  If $\rho \geq C_{\ref{thm:counting}} \max\{e^{-m/l}, m^{-l/4}\}$ and $p$ is an odd prime with $C_{\ref{thm:counting}} \rho^{-1} \leq p \leq 2^{n/m}$, then
	\[
	|\phi_p(\mathbf{S}_p)| \leq \left( \frac{5 np^2}{m} \right)^m + \left(\frac{C_{\ref{thm:counting}} \rho^{-1}}{\sqrt{m/l}} \right)^n
	\]  
\end{theorem}

\section{Structure of Null-vectors}
The goal of this section is to show that vectors near the kernel of $M_n$ are unstructured.  The proof of this result is inspired by the method in Section 3 of \cite{jain2021quantitative}, which establishes a similar result in the non-symmetric case.  
For the symmetric setting, it is more convenient to work with a regularized version of the small-ball probability.  We define,
\[
\rho_{\beta}(\Bv, r) = \inf_{I \subset \{1, \dots, n\}: |I| = 2 \lfloor \beta n\rfloor} \rho(\Bv_{I}, r)
\] 
where $\Bv_{I}$ denotes the vector in $\R^{|I|}$ formed from $\Bv$ by keeping only the entries indexed by $I$.    
This parameter captures the anti-concentration of small segments of $\Bv$ and emulates the regularized LCD from \cite{vershynin2014symmetric}. 
We choose $\beta = o(1)$ and will drop the floor functions on $\lfloor \beta n \rfloor$ as they are not crucial to the calculations.
In this notation, we suppress the dependence of $\rho$ on the random variables $\xi_1, \dots, \xi_n$.  
We call a vector \emph{poor} if $\rho_{\beta}(\Bv, \eta) \leq \alpha$ and \emph{rich} otherwise.  We denote the poor and rich vectors by $\BP_\eta(\alpha)$ and $\BR_{\eta}(\alpha)$ respectively.  When $\eta$ is clear from context, we will use the abbreviations $\BP(\alpha)$ and $\BR(\alpha)$.  We show that with high probability, all approximate null-vectors are poor.

\begin{proposition} \label{prop:rich}
	Let $\|M\| \leq \exp(n^{1/15})$.  
	For $\alpha  \geq  \exp(-n^{2/15})$ and any $\eta \leq (C_{\ref{prop:rich}}\|M\| \alpha^{-1} n^{7/6})^{-4 \frac{\log(\alpha^{-1})}{\log n}} $ we have that
	\[
	\P(\exists \Bv \in \BR(\alpha): \|M_n \Bv\|_2 \leq \eta ) \leq C_{\ref{prop:rich}} \exp(-c_{\ref{prop:rich}} n)
	\]
	where $C_{\ref{prop:rich}} >1$ and $c_{\ref{prop:rich}}$ depend only on $\xi$.  	
\end{proposition}

We prove Proposition \ref{prop:rich} in a series of steps.  
For notational convenience, we define
\[
b = \frac{1}{15} \text{ and } \beta = n^{-3 b}.
\]
Additionally, to streamline the argument, we assume that $\|M\| \geq C_{\ref{lem:opnorm}} \sqrt{n}$.  We sketch the necessary modifications when this condition is violated at the end of the section.
The first lemma is an observation of Tao and Vu \cite{tao2008circular}.
\begin{lemma} \label{lem:j}
	For any $\Bv \in \BR(\alpha)$, there exists $j \in \{0, 1, \dots, J(\alpha, n, T)\}$ ($T > 1$) such that 
	\[
	\rho_{\beta }(\Bv, \eta T^{j+1}) \leq n^{b/4} \rho_{\beta }(\Bv, \eta T^{j})
	\]
	where $J = 5 b^{-1} \frac{\log(\alpha^{-1})}{\log n}$ and $T = c_{\ref{lem:constant}}^{-1}\|M\| \alpha^{-1} n n^b \beta^{-1/2}$.
\end{lemma}
\begin{proof}
	Note that $\rho_{\beta}(\Bv,  \eta T^j)$ is an increasing function in $j$ and by definition all $\rho$ are bounded by 1.  Assume to the contrary that no such $j$ existed, then 
	\[
	\rho_{\beta} (\Bv, \eta T^{J}) \geq n^{Jb/4} \alpha  > 1,
	\]  
	a contradiction.
\end{proof}

Paying a price of $\binom{n}{2 \beta n} \ll \exp(n)$, we can assume that $\rho_{\beta}(\Bv, \eta T^j)$ is achieved by the first $2\beta n$ coordinates of $\Bv$.  We can also divide $\Bv$ into subvectors, $\Bv_1, \dots, \Bv_m$ of consecutive indices of size between $\beta n$ and $2 \beta n$, where $\Bv_1 \in \R^{2 \beta n}$ and $\rho(\Bv_1, \eta T^j) = \rho_{\beta}(\Bv, \eta T^j)$.  We allow for a range of sizes for divisibility issues.  Note that we use the same subdivision for all vectors with the same $2 \beta n$ coordinates that achieve the $\rho_\beta$ bound so that we still only need to take the $\binom{n}{2 \beta n} \ll \exp(n)$ union bound. We use $\dim(\Bv_i)$ to denote the size of the subvector.  For every $\Bv \in \BR(\alpha)$, we fix such a $j$ from Lemma \ref{lem:j} arbitrarily and we use $\BR_j(\alpha)$ to be those vectors in $\BR(\alpha)$ indexed by $j$.  This partitions 
\[
\BR(\alpha) = \sqcup_{j=0}^J \BR_j(\alpha).
\]
We divide this partition further by introducing
\[
\BR_{j,\ell} = \{\Bv \in \BR_j(\alpha): \rho_{\beta }(\Bv,\eta T^j) \in (2^{-\ell-1}, 2^{-\ell}]\}.
\]

We focus on a fixed $j$ and $\ell$ and take a union bound over $|J| \times \log(n/\alpha)$ at the end.

\begin{lemma} \label{lem:nearinteger}
	For $\Bv \in \BR_{j,\ell}$, there exist $D_1, \dots, D_{\beta^{-1}} \in [c_{\ref{lem:nearinteger}}\alpha, n^{b}]$ and some $\Bw_{i}' \in \Z^{\dim(\Bv_i)}$ such that 
	\[
	\|\Bw_i - \Bv_i'\| \leq n^{b}.
	\]
	where $\Bw_i = \eta^{-1} T^{-j} D_i \Bv_i$.
\end{lemma}

\begin{proof}
	Let $f(n) = c_{\ref{lem:nearinteger}}\alpha$, $g(n) = n^{b}$ and $\Bx = \eta^{-1} T^{-j} \Bv_i$. If the desired conclusion does not hold, then for all $t \in [f(n), g(n)]$ and $i$, $\dist(t \Bx_i, \Z^{\beta n}) \geq n^{b}$.  By Proposition \ref{prop:distance}, 
	\[
	\rho(\Bx_i, 1) \leq C_{\ref{prop:distance}} \exp(\pi)  \left(2\exp\left(-c_{\ref{prop:distance}} n^{2b}\right)  + c_{\ref{lem:nearinteger}}\alpha\right) < \alpha
	\]
	for small enough constant $c_{\ref{lem:nearinteger}} > 0$ since $\alpha \geq  2^{-n^{2b}}$.
	
	This implies that
	\[
	\rho(\Bv_i, \eta) \leq \rho(\Bv_i, T^j \eta) = 	\rho_{1, \xi}(\Bx_i) < \alpha
	\]
	which is a contradiction to our assumption that $\Bv \in \BR_{j,\ell}$ since $\rho(\Bv, \eta) \leq \rho(\Bv_i, \eta)$ by Lemma \ref{lem:restriction}.  
\end{proof}

The previous lemma allows us to reduce our claim to a statement about the integer lattice.  Let $\EE$ be the event that $\|N_n\| \leq C_{\ref{lem:opnorm}} \sqrt{n}$.  
Therefore, by Lemma \ref{lem:opnorm},
\[
\P(\exists \Bv \in \BR_{j, \ell}(\alpha): \|M_{n} \Bv\| \leq \eta) \leq \P(\exists \Bv \in \BR_{j, \ell}(\alpha): \|M_{n} \Bv\| \leq \eta \cap \EE ) + 2\exp(-n).
\]
Suppose the event on the right occurs, in other words, there is a $\Bv$ such that $\|M_{n} \Bv\| \leq \eta$. By Lemma \ref{lem:nearinteger}, there exist $D_i \in [c_{\ref{lem:nearinteger}}\alpha/n, n^{b}]$ and $\Bv' \in \Z^n$ such that 
\[
\| \Bw_i - \Bv_i'\| \leq n^{b}.
\]   
For $i \in [m]$, let $\mathbf{D}_i = D_i \I_{\dim(\Bv_i)}$. Then, let $D$ denote the diagonal matrix 
\[
D = \left( \begin{array}{cccc} \mathbf{D}_1  & &  & \\
	& \mathbf{D}_2 & & \\
 & & \ddots & \\
& & & \mathbf{D}_m\end{array} \right).
\]  
Thus,
\begin{align*}
	\|M_n \Bv'\|_2 &\leq \|M_n \eta^{-1} T^{-j} D \Bv\|_2 + \|M_n\| \|\Bw - \Bv'\|_2 \\
	&\leq \eta^{-1} T^{-j} \|D\|  \|M_n\Bv\|_2 +  (\|M\| +C_{\ref{lem:opnorm}} \sqrt{n}) \beta^{-1/2} n^{b}  \\
	&\leq T^{-j}\|D\| + 2\|M\| \beta^{-1/2}  n^{b} \\
	&\leq 3 \|M\|  \beta^{-1/2} n^{b}
\end{align*} 
where we have utilized the observations that $T > \|M\|$ and $\|M\|\geq C_{\ref{lem:opnorm}} \sqrt{n}$.

Now let $X_i$ denote the row of $M_n$.  For the event above to hold, there must be at least $n - n^{2b}$ coordinates for which \[
|X_i \cdot \Bv'| \leq 3 \|M\| \beta^{-1/2}.
\] 
Therefore, by Lemma \ref{lem:restriction}, 
\[
\P(\|M_n \Bv'\| \leq 3 \|M\| n^{b} \beta^{-1/2}) \leq \rho_{\beta}(\Bv', 3 \|M\| \beta^{-1/2})^{n - n^{2b}- \beta n}.
\]
This is the point where it is crucial that we use a regularized version of $\rho$, as this allows us to use the independence of $n- n^{2b} - \beta n$ rows to achieve the exponent on the right hand side.
Let us define 
\begin{equation} \label{eq:net}
\widetilde{\BR_{j,\ell}}(\alpha) = \{ \Bv' \in \Z^n: \exists \Bv \in \BR_{j, \ell}(\alpha), D_i \in \R \text{ s.t. } |D_i| \in [c_{\ref{lem:nearinteger}}\alpha/n, n^{b}], \|\eta^{-1} T^{-j}  D \Bv_i - \Bv_i'\| \leq n^{b}\}
\end{equation}

We need to control the entropy of $\widetilde{\BR_{j,\ell}}(\alpha)$.
\begin{proposition} With $\widetilde{\BR_{j,\ell}}(\alpha)$ defined in \eqref{eq:net}, we have that
	\[
	|\widetilde{\BR_{j,\ell}}(\alpha)| \leq 2^{2 n^{1-b} \beta^{-1}}+ \left( \frac{8 C_{\ref{thm:counting}}  2^{\ell} }{n^{b/2}} \right)^{ n}.
	\]
\end{proposition}
\begin{proof}
	Let $\Bv' \in \widetilde{\BR_{j,\ell}}(\alpha)$ and $\Bv \in \BR_{j, \ell}(\alpha)$ such that 
	\[
	\|\Bx_i- \Bv_i'\| \leq n^{b}
	\] 
	where $\Bx_i = \eta^{-1} T^{-j}  D_i \Bv_i$. 
	Thus,
	\begin{align*}
		\rho(\Bv_i', 2n^{2b}) &\geq \rho(\Bx_i, n^{2b}) - e \cdot \exp\left(- c_{\ref{lem:stability}} n^{2b} \right) \\
		&\geq \rho(\Bv_i, \eta T^j |D_i|^{-1} n^{2b}) - e \cdot \exp\left(- c_{\ref{lem:stability}} n^{2b} \right) \\
		&\geq \rho(\Bv_i, \eta T^j) - e \cdot \exp\left(- c_{\ref{lem:stability}} n^{2b} \right) \\
		&\geq  \rho(\Bv_i, \eta T^j) - e \cdot \exp\left(- c_{\ref{lem:stability}} n^{2b} \right) \\
		&\geq \frac{\rho(\Bv_i, \eta T^j)}{2}.
	\end{align*}
	In the third line, we have used the fact that $|D_i|^{-1} n^{2b} \geq 1$ and $e \cdot \exp\left(- c_{\ref{lem:stability}} n^{2b} \right) \leq \alpha$.  
	
	By the pigeonhole principle, 
	\[
	\rho(\Bv', 1) \geq \frac{\rho(\Bv_i', 2n^{2b})}{4 n^{2b}} \geq \frac{\rho(\Bv_i, \eta T^j)}{8n^{2b}} \geq \frac{2^{-\ell}}{8n^{2b}}.
	\]
	This implies that $\widetilde{\BR_{j,\ell}}(\alpha) \subset \BS_{\frac{2^{-\ell}}{8n^{2b}}}$ in the notation of Theorem \ref{thm:counting}.  Therefore, applying Theorem \ref{thm:counting}, with $m = n^{7b}$, $l = n^{2b} \log n$ and $p = 2^{n^{1-8b}}$, we have that
	\[
	|\varphi_p(\BS_{\frac{2^{-\ell}}{8n^{2b}}})| \leq \left( \frac{5\beta n 2^{2n^{1- 8b}}}{n^{7b}} \right)^{n^{7b}} + \left( \frac{8 C_{\ref{thm:counting}}  2^{\ell} }{n^{b/2}} \right)^{\beta n} \leq 2^{2 n^{1-b}}+ \left( \frac{8 C_{\ref{thm:counting}} 2^{\ell} }{n^{b/2}} \right)^{\beta n}
	\]  
	
	Note that 
	\begin{align*}
		\|\Bv'\|_\infty &\leq \eta^{-1} \|D\|_\infty \\
		&\leq T^{J+1} \exp(n^{2b}) \\
		&\leq (\|M\| \alpha^{-1} n^3)^{\log(\alpha^{-1})} \exp(n^{2b}) \\
		&\leq p
	\end{align*}
	since $\|M\| \leq \exp(n^{1-14b})$.  This implies that $\varphi_p$ is injective.  
	Finally, we take the product net over all $i$.
\end{proof}

We now control the small-ball probability for vectors in $\widetilde{\BR_{j,\ell}(\alpha)}$.

\begin{lemma}\label{lem:smallballR}
	For any $\Bv' \in \widetilde{\BR_{j,\ell}(\alpha)}$,
	\[
	\rho(\Bv', 3 \|M\| n^{b} \beta^{-1/2}) \leq \min\{1-c_{\ref{lem:constant}}, \}.
	\]
\end{lemma}

\begin{proof}  We begin with the first half of the inequality.
	\begin{align*}
		\rho(\Bv_i', 3 \|M\| n^{b} \beta^{-1/2}) &\leq \rho(\Bx_i, 4 \|M\| n^{b} \beta^{-1/2}) + e \cdot \exp\left(- \frac{c_{\ref{lem:stability}} \|M\|^2 n^{2b}  }{\beta n^{2b}} \right) \\
		&\leq \rho(\eta^{-1} T^{-j}  D_i \Bv_i, 4 \|M\| n^{b} \beta^{-1/2}) + e \cdot \exp\left(- \frac{c_{\ref{lem:stability}} \|M\|^2 }{\beta} \right) \\
		&\leq \rho( \Bv_i, 4 \|M\| n^{b} \beta^{-1/2} \eta T^{j}  |D_i|^{-1}) + e \cdot \exp\left(- \frac{c_{\ref{lem:stability}} \|M\|^2 }{\beta} \right) \\
		&\leq \rho( \Bv_i, 4 \|M\| n^{b} \beta^{-1/2} \eta T^{J}  |D_i|^{-1}) + e \cdot \exp\left(- \frac{c_{\ref{lem:stability}} \|M\|^2 }{\beta} \right) \\
		&\leq \rho( \Bv_i, 4 \|M\| n^{b} \beta^{-1/2} \eta T^{J}  c_{\ref{lem:nearinteger}}^{-1} \alpha^{-1} n) + e \cdot \exp\left(- \frac{c_{\ref{lem:stability}} \|M\|^2 }{\beta} \right) \\
		&\leq 1 - c_{\ref{lem:constant}}
	\end{align*}
	Here we are using that $4 \|M\| n^{b} \beta^{-1/2} \eta T^{J} \alpha^{-1} n < T^{J+1} c_{\ref{lem:constant}}$.  This is due to our choice of defining $T = c_{\ref{lem:constant}}^{-1} \|M\| \alpha^{-1} n n^b \beta^{-1/2}$ and by assumption, $\eta \leq T^{-(J+1)}$.   
	
	On the other hand,
	\begin{align*}
		\rho(\Bv_i', 3 \|M\| n^{b} \beta^{-1/2}) &\leq \rho( \Bx_i, 4 \|M\| n^{b} \beta^{-1/2}) + e \cdot \exp\left(- \frac{c_{\ref{lem:stability}} \|M\|^2 n^{2b} }{\beta n^{2b}} \right) \\
		&\leq \rho( \eta^{-1} T^{-j}  D_i \Bv_i, 4 \|M\| n^{b} \beta^{-1/2}) + e \cdot \exp\left(- \frac{c_{\ref{lem:stability}} \|M\|^2  }{\beta } \right) \\
		&\leq \rho(  \Bv_i, 4 \|M\| n^{b} \beta^{-1/2} \eta T^{j}  |D_i|^{-1}) + e \cdot \exp\left(- \frac{c_{\ref{lem:stability}} \|M\|^2  }{\beta } \right) \\ 
		&\leq \rho(  \Bv_i, \eta T^{j+1}) + e \cdot \exp\left(- \frac{c_{\ref{lem:stability}} \|M\|^2  }{\beta } \right) \\ 
		&\leq n^{b/4} \rho( \Bv_i, \eta T^{j})  + e \cdot \exp\left(- \frac{c_{\ref{lem:stability}} \|M\|^2  }{\beta } \right) \\ 
		&\leq 2 n^{b/4} 2^{-\ell}
	\end{align*}
\end{proof}

We now have all the ingredients to prove Proposition \ref{prop:rich}.
\begin{proof}[Proof of Proposition \ref{prop:rich}]  By the deductions above, 
	\begin{align*}
		\P(\exists \Bv \in &\BR_{j, \ell}(\alpha): \|M_n \Bv\|_2 \leq \eta) \leq \sum_{\Bv' \in \widetilde{\BR_{j,\ell}(\alpha)}} \rho(\Bv_i', 3 \|M\| n^{b} \beta^{-1/2})^{n - \beta n - n^{2b}} + 2 \exp(-n) \\
		&\leq |\widetilde{\BR_{j,\ell}(\alpha)}| \left(\min\{1-c_{\ref{lem:constant}}, 2 n^{b/4} 2^{-\ell}\} \right)^{n - \beta n - n^{2b}}  + 2 \exp(-n) \\
		&\leq C_{\ref{thm:counting}}\left( 2^{2 n^{1-b} \beta^{-1}}+ \left( \frac{8 C_{\ref{thm:counting}} 2^{\ell} }{n^{b/2}} \right)^{ n}\right) \left(\min\{1- c_{\ref{lem:constant}}, 2 n^{b/4} 2^{-\ell}\} \right)^{n - \beta n - n^{b/2}} + 2 \exp(-n) \\
		&\leq C_{\ref{thm:counting}} \left(  2^{2 n^{1-b} \beta^{-1}} (1 - c_{\ref{lem:constant}})^{n - \beta n - n^{2b}} + \left( \frac{8 C_{\ref{thm:counting}}  2^{\ell} }{n^{b/2}} \right)^{ n} (2 n^{b/4} 2^{-\ell})^{n - \beta n - n^{b/2}}\right) + 2 \exp(-n)\\
		&\leq \left( (16 C_{\ref{thm:counting}})^n  n^{-bn/4} \alpha^{-\beta n - n^{2b}} \right) + \Omega( \exp(-n)) \\
		&\leq (16 C_{\ref{thm:counting}})^n \exp(-bn/4 \log n + n^{2 b} \beta n +  n^{3b/2}) + \Omega( \exp(-n)) \\
		&\leq O(\exp(-\Omega(n)) ).
	\end{align*}
	The last line holds due to our setting $\beta = n^{-3b}$.  
	
	It remains to prove this result when $\|M\| \leq C_{\ref{lem:opnorm}} \sqrt{n}$.  In fact, this case is significantly easier.  As now $\|M_n\|$ is on the order of $\|N_n\|$ one can replace our definition of $T$ with $T' = c_{\ref{lem:constant}}^{-1} \alpha^{-1} n^{3/2} n^b \beta^{-1/2} $ where  we have simply replaced $\|M\|$ with $\sqrt{n}$.  The identical argument above then goes through.  We omit the details. 
\end{proof}

\section{Structure of Eigenvectors }
In this section, we prove the following theorem about the eigenvectors of $M_n$.

\begin{theorem} \label{thm:eigenvector}
	Let $\|M\| \leq \exp(n^{1/16})$.  
	For $\alpha  \geq  \exp(-n^{2/15})$ and $\nu = (C_{\ref{prop:rich}}\|M\| \alpha^{-1} n^{7/6})^{-4 \frac{\log(\alpha^{-1})}{\log n}}$ we have that
	\[
	\P(\exists \text{ an eigenvector of } M_n \text{ in }\BR_{\nu}(\alpha)  ) \leq C_{\ref{thm:eigenvector}} \exp(-c_{\ref{thm:eigenvector}} n)
	\]
	where $C_{\ref{thm:eigenvector}} >1$ and $c_{\ref{thm:eigenvector}}$ depend only on $\xi$.  
\end{theorem}

\begin{proof}
Let $\EE$ be the event that $\| M_n \| \leq \|M\| + C_{\ref{lem:opnorm}} \sqrt{n}$.  By Lemma \ref{lem:opnorm},
\[
\P(\EE) \geq 1 - 2 \exp(-n).
\]
Note that for $\lambda \in \R$ with $|\lambda| \leq \|M\| + C_{\ref{lem:opnorm}} \sqrt{n}$,
\[
\| M_n + \lambda\| \leq 2 \exp(n^{1/16}) + 2 C_{\ref{lem:opnorm} }\sqrt{n} \leq \exp(n^{1/15}).
\]
Thus, we may apply Proposition \ref{prop:rich} to conclude that 
any vector $\Bv$ such that $ \| (M_n - \lambda) v\|_2 \leq \nu$ has $\rho(\Bv) \leq \alpha$ with probability at least $1 - \exp(-n)$.  
As this applies to any $\lambda \in [-\|M\| -C_{\ref{lem:opnorm}} \sqrt{n}, -\|M\| + C_{\ref{lem:opnorm}}]$, we can deduce the structure of an eigenvector via a net argument. 
We use an $\nu$ net of the interval $[-\|M\| -C_{\ref{lem:opnorm}} \sqrt{n}, -\|M\| + C_{\ref{lem:opnorm}}]$ which can be chosen to have size $2 (\|M\| + C_{\ref{lem:opnorm}} \sqrt{n}) \nu^{-1} = o(\exp(n))$. On the event $\mathcal{E}$, for any eigenvector $v$ of $M_n$ and corresponding eigenvalue $\lambda$, we let $\lambda_0$ be the nearest point in the net to $\lambda$.  We then have that
\[
\|(M_n - \lambda_0) v\|_2 \leq \|(M_n - \lambda) v\| + \|(\lambda - \lambda_0) v\| \leq \nu. 
\]      
Therefore, by a simple union bound over the points in the net, we conclude that no eigenvector can reside in $\R(\alpha)$.   
\end{proof}

\section{Eigenvalue Gaps}
Now we complete the proof of Theorem \ref{thm:main}.  Using the notation of Section \ref{sec:strategy}, we have that 
\[
\P(|\lambda_{i+1}(M_{n}) - \lambda_i(M_n)| \leq \nu/\sqrt{n}) \leq n \P(\mathbf{w}^T \mathbf{x} \leq \nu).
\] 
By Theorem \ref{thm:eigenvector}, if we let $\alpha \geq  \exp(-n^{2/15})$ and  $\nu = (C_{\ref{prop:rich}}\|M\| \alpha^{-1} n^{7/6})^{-4 \frac{\log(\alpha^{-1})}{\log n}}$ then
\[
\rho(\mathbf{w}, \nu) \leq \alpha
\]
with probability at least $1 - C_{\ref{thm:eigenvector}} \exp(-c_{\ref{thm:eigenvector}} n)$.  As $\Bx$ is independent of $\Bw$, we find that
\[
\P(|\lambda_{i+1}(M_{n}) - \lambda_i(M_n)| \leq \nu/\sqrt{n}) \leq n (\alpha + C_{\ref{thm:eigenvector}} \exp(-c_{\ref{thm:eigenvector}} n)).
\] 
One final union bound over $i$ yields the tail bound for the smallest eigenvalue gap and  concludes the proof upon adjusting $C_{\ref{thm:main}}$.

\section*{Acknowledgements}
	The authors would like to thank Spencer Shortt and Vishesh Jain for helpful discussions.

\bibliographystyle{plain}
\bibliography{perturbed}

\end{document}